\documentclass[12pt,reqno]{amsart}                  
\usepackage{amsfonts}
\usepackage{amsmath}
\usepackage{amsthm}
\usepackage{mathrsfs}
\usepackage{mathbbol}

\newtheorem{theorem}{Theorem}
\newtheorem{lemma}{Lemma}
\newcommand{\Tr}{{\rm Tr\,}}
\theoremstyle{definition}
\newtheorem{definition}{Definition}
\begin{document}
\title[Verification by stochastic Perron's method]{Verification by stochastic Perron's method in stochastic exit time control problems}

\author{Dmitry B. Rokhlin}

\address{D.B. Rokhlin,
Faculty of Mathematics, Mechanics and Computer Sciences,
              Southern Federal University, 
Mil'chakova str., 8a, 344090, Rostov-on-Don, Russia}          
\email{rokhlin@math.rsu.ru}                   

\begin{abstract} We apply the Stochastic Perron method, created by Bayraktar and S\^irbu, to a stochastic exit time control problem. Our main assumption is the validity of the Strong Comparison Result for the related Hamilton-Jacobi-Bellman (HJB) equation. Without relying on Bellman's optimality principle we prove that inside the domain the value function is continuous and coincides with a viscosity solution of the Dirichlet boundary value problem for the HJB equation.
\end{abstract}

\maketitle

\section{Introduction} 
\label{sec:1}
As is known, the value function of a controlled diffusion problem can be characterized as a viscosity solution of the corresponding Hamilton-Jacobi-Bellman (HJB) equation. Usually a related reasoning is based on Bellman's dynamic programming principle (DPP). However, a proof of this principle may be a separate difficult problem. Recently Bayraktar and S\^irbu \cite{BaySir12a}, \cite{BaySir12b} proposed an alternative approach, whose starting point is the notions of stochastic sub- and supersolutions, estimating the value function $v$ from below and above:
\begin{equation} \label{eq:1.1}
u\le v\le w.
\end{equation}
The exact definitions depend on the problem, but the idea (going back to Stroock and Varadhan \cite{StrVar72}) is to generate sub- and supermartingale-like processes by the superposition of $u$ and $w$ with the state process. 

The subsequent argumentation can be loosely described as follows. Denote by $\mathcal V^{-}$, $\mathcal V^{+}$ the sets of stochastic sub- and supersolutions. From (\ref{eq:1.1}) we get
\begin{equation} \label{eq:1.2}
u_{-}(x)=\sup_{u\in\mathcal V^{-}} u(x)\le v(x)\le\inf_{w\in\mathcal V^{+}} w(x)=w_{+}(x).
\end{equation}
The main point is to prove that $u_-$ is a viscosity supersolution and $w_{+}$ is a viscosity subsolution of the boundary value problem for the associated HJB equation. Then, a comparison result ensures the reverse inequality $u_{-}\ge w_{+}$, and we can conclude that a unique (continuous) viscosity solution coincides with $v$.

The set  $\mathcal V_{-}$ is directed upward, that is, it is closed with respect to the pointwise maximum operation. Similarly, $\mathcal V_{+}$ is directed downward, and the whole scheme reminds the classical Perron method from the theory of harmonic functions. For the viscosity solutions the Perron method was developed by Ishii \cite{Ish87} (see also \cite{CraIshLio92}). But Ishii's construction is based on viscosity (not stochastic) semi-solutions and do not provide  inequalities like (\ref{eq:1.2}) for the value function $v$ of the controlled diffusion problem.

The Stochastic Perron Method (SPM) works under the severe assumption that a comparison result holds true. However, exactly this assumption ensures the uniqueness of a viscosity solution and the convergence of monotone, stable and consistent approximation schemes \cite{BarSou91}, \cite{Jak10}. 
The SPM can be regarded as a verification method, since it allows to prove that a viscosity solution of the HJB equation is the value function. The required comparison result is much weaker than the smoothness assumption in the usual verification scheme.

The creators of the SPM applied it to linear parabolic equations \cite{BaySir12a}, stochastic differential games \cite{BaySir12c}, \cite{Sir13} and finite horizon stochastic control problems, where the state process does not hit the boundary \cite{BaySir12b}. However, it is conjectured in \cite{Sir13} that ''any stochastic optimization problem could be treated using Stochastic Perron's Method, provided that it is properly formulated, and the stochastic semi-solutions defined accordingly''. 

In the present paper we are interested in the application of these ideas to stochastic exit time control problems. For such problems in general it is difficult to prove the DPP principle without additional assumptions. We only mention classical results \cite[Chapter III, Theorem 1.1]{Bor89}, \cite[Chapter V, Theorem 2.1]{FleSon06} and the paper \cite{DonKry07}. Due to the importance of stochastic exit time control problems for applications it is desirable to take a view from another perspective. 

In Section \ref{sec:2} we introduce the stochastic optimal control problem, recall the definitions of viscosity sub- and supersolutions and formulate the main result (Theorem \ref{th:1}), identifying the value function with a viscosity solution of the Dirichlet boundary value problem for the HJB equation. In Sections \ref{sec:3}, \ref{sec:4} we introduce the notions of stochastic sub- and supersolutions and obtain the estimates (\ref{eq:1.2}), closely following the ideas of \cite{BaySir12b}. Finally, in Section 5 we finish the proof of Theorem 1 and give a short proof of the DPP under the adopted assumptions.

\section{The exit time control problem and its value function}
\label{sec:2}
\setcounter{equation}{0}
Consider an $n$-dimensional controlled diffusion process $X$, governed by the system of stochastic differential equations  
\begin{equation} \label{eq:2.1}
dX_s=b(X_s,\alpha_s)ds+\sigma(X_s,\alpha_s)dW_s,\ \ X_0=x,
\end{equation}
where $W$ is a standard $m$-dimensional Brownian motion with respect to some filtration $\mathbb F=(\mathscr F_s)_{s\ge 0}$, satisfying the usual conditions. The control process $\alpha$ takes values in a compact subset $A$ of $\mathbb R^k$ and is assumed to be progressively measurable with respect to $\mathbb F$. Denote by $\mathcal A$ the set of such controls. We assume that the stochastic control problem is in the standard form (see \cite[Chapter 3]{Tou13}), that is, the drift vector $b$ and the diffusion matrix $\sigma$ are continuous on $\mathbb R^d\times A$ and satisfy Lipschitz and linear growth conditions:
\begin{align*}
|b(x,a)-b(y,a)|+|\sigma(x,a)-\sigma(y,a)| & \le K |x-y|, \\
|b(x,a)|+|\sigma(x,a)| & \le K (1+|x|)
\end{align*}
with some constant $K$ independent of $x$, $y$, $a$. These conditions ensure the existence of a unique strong solution of (\ref{eq:2.1}) on: see $[0,\infty)$ \cite[Chapter 2, Sect. 5]{Kry77}, 
\cite[Theorem 2.3]{Tou13}. This solution will be denoted by $X^{x,\alpha}$.

Let $G$ be a connected open set in $\mathbb R^d$. Consider a Borel set $\widehat G$ between $G$ and its closure: $G\subseteq\widehat G\subseteq\overline G$, and denote by  
$$\sigma^{x,\alpha}=\inf\{s\ge 0:X_s^{x,\alpha}\not\in\widehat G\}$$
the exit time of $X^{x,\alpha}$ from $\widehat G$. The value function $v$ of the corresponding stochastic control problem is defined as follows
\begin{equation} \label{eq:2.2}
v(x)=\sup_{\alpha\in\mathcal A} J(x,\alpha)=\sup_{\alpha\in\mathcal A}\mathsf E\left( \int_0^{\sigma^{x,\alpha}} e^{-\beta s} f(X_s^{x,\alpha},\alpha_s)\,ds+e^{-\beta \sigma^{x,\alpha}}g\left(X_{\sigma^{x,\alpha}}^{x,\alpha}\right)\right).
\end{equation}
Here $\beta>0$ and the functions 
$$f:\overline G\times A\mapsto\mathbb R,\ \ g:\partial G\mapsto\mathbb R$$
are continuous and bounded. Clearly, the value function $v$ is also bounded on $\overline G$.

From the theory of stochastic optimal control (see \cite{Lio83a}, \cite{Lio83b}, \cite{FleSon06}) it is known  that $v$ should satisfy in the ''viscosity sense'' the HJB equation
\begin{equation} \label{eq:2.3}
\beta v(x) - H(x,v_x(x),v_{xx}(x))=0,\ \  x\in G
\end{equation}
with the (continuous) Hamiltonian
$$ H(x,p,M)=\sup_{a\in A} \left[f(x,a)+b(x,a)\cdot p + \frac{1}{2}\Tr(\sigma(x,a)\sigma^T(x,a)M))\right],$$
$p\in\mathbb R^d$, $M\in\mathbb R^d\times\mathbb R^d$, together with the Dirichlet boundary condition 
\begin{equation} \label{eq:2.4}
v(x)=g(x),\ \ x\in\partial G. 
\end{equation}
We use the notation $v_x=(v_{x_i})_{i=1}^n$, $v_{xx}=(v_{x_i x_j})_{i,j=1}^n$ for the gradient vector and the Hessian matrix. 

Let us recall the corresponding definitions \cite{CraIshLio92}, \cite{BarBur95}, \cite{Jak10}. 
\begin{definition} \label{df:1}
A bounded upper semicontinuous (usc) function $u$ is called a \emph{viscosity subsolution} of (\ref{eq:2.3}), (\ref{eq:2.4}) if for any $\varphi\in C^2(\overline G)$ and for any local maximum point $x_0$ of $u-\varphi$ on $\overline G$, we have
\begin{align*}
\beta u(x_0) - H(x_0,\varphi_x(x_0),\varphi_{xx}(x_0)) & \le 0,\ \ x_0\in G,\\
\min\{\beta u(x_0) - H(x_0,\varphi_x(x_0),\varphi_{xx}(x_0)),u(x_0)-g(x_0)\} & \le 0,\ \ x_0\in\partial G.
\end{align*}

A bounded lower semicontinuous (lsc) function $w$ is called a \emph{viscosity supersolution} of (\ref{eq:2.3}), (\ref{eq:2.4}) if for any $\varphi\in C^2(\overline G)$ and for any local minimum point $x_0$ of $w-\varphi$ on $\overline G$, we have
\begin{align*}
\beta w(x_0) - H(x_0,\varphi_x(x_0),\varphi_{xx}(x_0)) & \ge 0,\ \ x_0\in G,\\
\max\{\beta w(x_0) - H(x_0,\varphi_x(x_0),\varphi_{xx}(x_0)),w(x_0)-g(x_0)\} & \ge 0,\ \ x_0\in\partial G.
\end{align*}

A bounded function $v$ is called a \emph{viscosity solution} of (\ref{eq:2.3}), (\ref{eq:2.4}) if its usc and lsc envelopes:
$$ v^*(x)=\inf_{\varepsilon>0}\sup\{ v(y): y\in B_\varepsilon(x)\cap\overline G\},\ \ 
   v_*(x)=\sup_{\varepsilon>0}\inf\{ v(y): y\in B_\varepsilon(x)\cap\overline G\}$$
are respectively viscosity sub- and supersolutions. Here $B_\varepsilon(x)$ is the open ball in $\mathbb R^n$ centered at $x$ with radius $\varepsilon$.
\end{definition}

In these definitions one can assume that the maximum (resp., minimum) point $x_0$ is strict and $\varphi(x_0)=u(x_0)$ (resp., $\varphi(x_0)=w(x_0)$). 

\begin{definition} \label{df:2} Following \cite{BarBur95}, \cite{BarRou98}, we say that the Strong Comparison Result (SCR) for (\ref{eq:2.3}), (\ref{eq:2.4}) holds true if $u\le w$ on $G$ for any subsolution $u$ and any supersolution $w$.
\end{definition}

Since $v^*\ge v_*$ this result implies that any viscosity solution $v$ is continuous on $G$. Moreover, any two viscosity solutions may differ only at $\partial G$. The validity of the SCR for stochastic exit time control problems in smooth bounded domains was studied thoroughly in \cite{BarBur95}, \cite{BarRou98}. The case of non-smooth boundary, including the case of finite horizon parabolic problems, was considered in \cite{Cha04}. Roughly speaking, it was proved that some non-degeneracy assumptions, concerning the boundary points, are enough. 

\begin{theorem} \label{th:1}
Assume that the SCR is satisfied. Then $v$ is continuous on $G$ and $v(x)=\tilde v(x)$, $x\in G$ for any viscosity solution $\tilde v$ of (\ref{eq:2.3}), (\ref{eq:2.4}).
\end{theorem}

Note that, under the SCR, a change of $\widehat G$ does not affect $v$ except of boundary points. 
Our aim is to give a proof of Theorem \ref{th:1} by the means of the SPM, as it was done in \cite{BaySir12a}, \cite{BaySir12c}, \cite{Sir13}, \cite{BaySir12b} for other problems, mentioned in the introductory section. 

\section{Stochastic subsolutions}
\label{sec:3}
\setcounter{equation}{0}
Consider the stochastic differential equation, obtained from (\ref{eq:2.1}) by a randomization of the initial condition:
\begin{equation} \label{eq:3.1}
X_t=\xi I_{\{t\ge\tau\}}+\int_\tau^t b(X_s,\alpha_s)\,ds+\int_\tau^t\sigma(X_s,\alpha_s)\,dW_s,\ \ t\ge 0.
\end{equation}
Here $\tau:\Omega\mapsto [0,\infty]$ is a stopping time with respect to $\mathbb F$ and $\xi$ is an $\mathscr F_\tau$-measurable random vector such that $\xi I_{\{\tau<\infty\}}$ is bounded and $\xi\in\overline G$ on $\{\tau<\infty\}$. Such a pair $(\tau,\xi)$ will be called a \emph{randomized initial condition}. By $\int_\tau^t(\cdot)$ we always mean $\int_0^t I_{\{s\ge\tau\}}(\cdot)$. 

By the standard results (see \cite[Chapter 2, Sect.\,5]{Kry77}) there exists a pathwise unique strong solution $X^{\tau,\xi,\alpha}$ of (\ref{eq:3.1}). Note, that the values of $\xi$ on the set $\{\tau=\infty\}$ do not affect $X^{\tau,\xi,\alpha}$. The trajectories of the process $X^{\tau,\xi,\alpha}$ are continuous on the stochastic interval 
$$\Lbrack\tau,\infty\Rbrack=\{(\omega,t)\in\Omega\times [0,\infty): \tau(\omega)\le t\}.$$
In addition, $X^{\tau,\xi,\alpha}=0$ on $\Lbrack 0,\tau\Lbrack=\{(\omega,t)\in\Omega\times [0,\infty): t<\tau(\omega)\}$ and 
$$X_\tau^{\tau,\xi,\alpha}=\lim_{t\searrow\tau} X_t^{\tau,\xi,\alpha}=\xi\ \ \textrm{on}\ \{\tau<\infty\}.$$
For concreteness, for each component $X^{i,\tau,\xi,\alpha}$ of $X^{\tau,\xi,\alpha}$ we put 
$$X^{i,\tau,\xi,\alpha}_\infty=\liminf_{s\to\infty} X^{i,\tau,\xi,\alpha}_s.$$ 
The exit time of $X^{\tau,\xi,\alpha}$ is defined by
$$\sigma^{\tau,\xi,\alpha}=\inf\{t\ge\tau:X_t^{\tau,\xi,\alpha}\not\in \widehat G\}.
$$
To reconcile this notation with Section 2, we put $X^{x,\alpha}=X^{0,x,\alpha}$, $\sigma^{x,\alpha}=\sigma^{0,x,\alpha}$.

Let $u$ be a uniformly bounded continuous function: $u\in C_b(\overline G)$. The process
$$ Z_t^{\tau,\xi,\alpha}(u)=\int_\tau^t e^{-\beta s} f(X_s^{\tau,\xi,\alpha},\alpha_s)\,ds+I_{\{t\ge\tau\}}e^{-\beta t} u(X_t^{\tau,\xi,\alpha}) $$
plays a major role in the definition of stochastic semi-solutions. 

\begin{definition} \label{df:3}
Let us say that a control process $\alpha\in\mathcal A$ is $u$-\emph{suitable} for $(\tau,\xi)$ if	
\begin{equation} \label{eq:3.2}
\mathsf E(Z_\rho^{\tau,\xi,\alpha}(u)|\mathscr F_\tau)\ge Z_\tau^{\tau,\xi,\alpha}(u)
\end{equation} 
for any stopping time $\rho\in [\tau,\sigma^{\tau,\xi,\alpha}]$. A function $u\in C_b(\overline G)$ is called a \emph{stochastic subsolution} of (\ref{eq:2.3}), (\ref{eq:2.4}) if $u\le g$ on $\partial G$ and for any randomized initial condition $(\tau,\xi)$ there exists a $u$-suitable control $\alpha\in\mathcal A$.
\end{definition}

The set of stochastic subsolutions is denoted by $\mathcal V^{-}$.
Condition (3.2), which is weaker than the submartingale property, was introduced in \cite{BaySir12b}.

For any stochastic subsolution $u$ we have the inequality $u\le v$ on $\overline G$, where $v$ is defined by (\ref{eq:2.2}). Indeed, put in $\tau=0$, $\xi=x\in\overline G$, take a corresponding $u$-suitable control process $\alpha$, and put $\rho=\sigma^{x,\alpha}$. Then by the definitions, with the convention $Z^{x,\alpha}=Z^{0,x,\alpha}$,  we have
\begin{align*}
u(x)= Z_0^{x,\alpha}(u)\le \mathsf E Z_{\sigma^{x,\alpha}}^{x,\alpha}(u)\le J(x,\alpha)\le v(x).
\end{align*}
The last but one inequality follows from the fact that $X^{x,\alpha}_{\sigma^{x,\alpha}}\in\partial G$ on the set $\{\sigma^{x,\alpha}<\infty\}$ and $u\le g$ on $\partial G$.

The set $\mathcal V^{-}$ is non-empty, as it contains a constant
$$c\le\min\{\underline f/\beta,\underline g\},\ \ \underline f=\inf_{(x,a)\in\overline G\times A} f(x,a),\ \ \underline g=\inf_{x\in\partial G} g(x).$$
This assertion follows from the inequality
$$ \mathsf E(Z_\rho^{\tau,\xi,\alpha}(c)|\mathscr F_\tau)\ge\mathsf E\left[(\underline f/\beta-c)(e^{-\beta\tau}-e^{-\beta\rho})+ c e^{-\beta\tau}|\mathscr F_\tau\right]\ge c e^{-\beta\tau}=Z_\tau^{\tau,\xi,\alpha}(c),$$
showing that any $\alpha\in\mathcal A$ is $c$-suitable.

We will occasionally use the notation $\tau_A=\tau I_{A}+(+\infty) I_{A^c}$.

\begin{lemma} \label{lem:1}
Let $u_1, u_2$ be stochastic subsolutions. Then $u_1\vee u_2$ is a stochastic subsolution.
\end{lemma}
\begin{proof}
Let $\alpha^i\in\mathcal A$, $i=1,2$ be $u_i$-suitable controls for a randomized initial condition $(\tau,\xi)$. Put $A=\{u_1(\xi)>u_2(\xi)\}\in\mathscr F_\tau$ and define $\alpha\in\mathcal A$ by
$$ \alpha=I_A I_{\{\tau\le t\}}\alpha^1+I_{A^c} I_{\{\tau\le t\}}\alpha^2.$$
From the pathwise uniqueness property it easily follows that 
$$X^{\tau,\xi,\alpha}=X^{\tau,\xi,\alpha^1}\ \textnormal{on}\ \Lbrack\tau_A,\rho_A\Rbrack, \ \ \
  X^{\tau,\xi,\alpha}=X^{\tau,\xi,\alpha^2}\ \textnormal{on}\ \Lbrack\tau_{A^c},\rho_{A^c}\Rbrack. $$
Thus,
\begin{align*}
\mathsf E(Z_\rho^{\tau,\xi,\alpha}(u)|\mathscr F_\tau) & = \mathsf E(I_A Z_\rho^{\tau,\xi,\alpha^1}(u)+I_{A^c} Z_\rho^{\tau,\xi,\alpha^2}(u)|\mathscr F_\tau)\\
&\ge I_A \mathsf E(Z_\rho^{\tau,\xi,\alpha^1}(u_1)|\mathscr F_\tau)+I_{A^c} \mathsf E(Z_\rho^{\tau,\xi,\alpha^2}(u_2)|\mathscr F_\tau)\\
&\ge I_A Z_\tau^{\tau,\xi,\alpha^1}(u_1) + I_{A^c} Z_\tau^{\tau,\xi,\alpha^2}(u_2)\\
&=e^{-\beta\tau} u_1(\xi) I_A+e^{-\beta\tau} u_2(\xi) I_{A^c}=e^{-\beta\tau} u(\xi)=Z_\tau^{\tau,\xi,\alpha}(u)
\end{align*}
and $\alpha$ is $u$-suitable for $(\tau,\xi)$.
\end{proof}

In the proof of the next lemma we apply a result on approximation of a lower semicontinuous function by an increasing sequence of continuous ones: Theorem 1.3.7 of \cite{Bee93}. In \cite{BaySir12b} Proposition 4.1 of \cite{BaySir12a} is used instead.
\begin{lemma} \label{lem:2}
There exists a sequence $u_n\in\mathcal V^-$, $u_n(x)\le u_{n+1}(x)$, $x\in\overline G$ such that
$$\lim_{n\to\infty} u_n(x)=u_-(x):=\sup\limits_{u\in\mathcal V^-} u(x).$$
\end{lemma}
\begin{proof} Since $u_-$ is a lower semicontinuous function, there exists an  increasing sequence $g_n\in C(\overline G)$ such that $g_n(x)\nearrow u_-(x)$: see \cite[Theorem 1.3.7]{Bee93}. For any $x\in\overline G$ take $u_{n,x}\in\mathcal V^-$:
$$ g_n(x)<u_{n,x}(x)+1/n.$$
The sets
$ A_{x,n}=\{y\in\overline G: g_n(y)< u_{n,x}(y)+1/n\}$
are open in the relative topology of $\overline G$. For any compact subset $K_n$ of $\overline G$ there exists a finite subcover
$$ \bigcup_{i=1}^{N_n} A_{x_i,n}\supset K_n,\ \ x_i\in K_n.$$

Take a sequence of compact sets $K_n\subset K_{n+1}$ such that $\cup_{i=n}^\infty K_n =\overline G$. For any $u_0\in\mathcal V^-$ define $u_n\in\mathcal V^-$ by the recurrence relation
$$ u_n=\bigvee_{i=1}^{N_n} u_{x_i,n}\bigvee u_{n-1}. $$
We have $g_n<u_n+1/n$ on $K_n$. For $x\in\overline G$ take $N(x)$ such that $x\in K_n$, $1/n<\varepsilon/2$ and $u_-(x)-g_n(x)<\varepsilon/2$ for $n\ge N(x)$. Then
\[u_-(x)-u_n(x)<u_-(x)-g_n(x)+1/n<\varepsilon,\ \ n\ge N(x). \qedhere \]
\end{proof}

The following assertion is the most important part of the SPM.
\begin{theorem} \label{th:2}
The function
$$ u_-(x)=\sup\limits_{u\in\mathcal V^-} u(x)$$
is a viscosity supersolution of (\ref{eq:2.3}), (\ref{eq:2.4}).
\end{theorem}
\begin{proof}
Let $x_0\in\partial G$. If $u_-$ is not a viscosity supersolution then there exist $\varphi\in C^2(\overline G)$ and $\varepsilon>0$ such that $u_-(x_0)=\varphi(x_0)$, $u_->\varphi$ on the set $(\overline{B_\varepsilon (x_0)}\backslash\{0\})\cap\overline G$  and
$$ \max\{\beta \varphi(x_0) - H(x_0,\varphi_x(x_0),\varphi_{xx}(x_0)),\varphi(x_0)-g(x_0)\} < 0. $$
With the notation 
\begin{equation} \label{eq:3.3}
(\mathcal L^a\varphi)(x)=b(x,a)\varphi_x(x)+\frac{1}{2}\Tr\left(\sigma(x,a)\sigma^T(x,a)\varphi_{xx}(x)\right)
\end{equation}
we have $\varphi(x_0)<g(x_0)$ and
$$ \beta\varphi(x_0)-(\mathcal L^a\varphi)(x_0)-f(x_0,a)<0$$
for some $a\in A$. By the continuity of $b$, $\sigma$, $f$, $g$ one can take $\varepsilon>0$ such that
\begin{equation} \label{eq:3.4}
\varphi(x)<g(x)\ \ \textnormal{on}\ B_\varepsilon(x_0)\cap \partial G,
\end{equation}
\begin{equation} \label{eq:3.5}
\beta\varphi(x)-(\mathcal L^a\varphi)(x)-f(x,a)<0\ \ \textnormal{on}\ B_\varepsilon(x_0)\cap \overline G.
\end{equation}

Put $S_\varepsilon=\left(\overline B_\varepsilon(x_0)\backslash B_{\varepsilon/2}(x_0)\right)\cap\overline G$. By the lower semicontinuity of $u_-$ we have
$$ u_-(x)-\varphi(x)\ge\delta>0,\ \ x\in S_\varepsilon.$$
According to Lemma \ref{lem:2} there exists an increasing sequence $u_n\in\mathcal V^-$, $u_n\nearrow u_-$. The sets
$$ A_n=\{x\in S_\varepsilon:u_n(x)-\varphi(x)\le\delta'\},\ \ \delta'\in (0,\delta) $$
are compact, $A_n\supset A_{n+1}$ and 
$\cap_{n=1}^\infty A_n=\emptyset$ since $u_n\nearrow u_-$. Thus, $\cap_{n=1}^N A_n=\emptyset$ for some $N$. This means that there exists $u=u_N\in\mathcal V^-$ such that $u-\varphi>\delta'$ on $S_\varepsilon$. 

Now define the function $\varphi^\eta=\varphi+\eta$, where $\eta\in (0,\delta')$ is such that inequalities (\ref{eq:3.4}), (\ref{eq:3.5}) hold true for $\varphi^\eta$ instead of $\varphi$. Note that
$$ u-\varphi^\eta=u-\varphi-\eta>\delta'-\eta>0\ \ \textnormal{on}\ S_\varepsilon.$$
We claim that
$$ u^\eta=\left\{\begin{array}{l l}
    \varphi^\eta\vee u  & \text{on}\quad B_\varepsilon(x_0)\cap \overline G,\\
    u & \ \text{otherwise} 
  \end{array} \right. $$
is a stochastic subsolution. This is the last (but the main) step of the proof since it gives a contradiction with the definition of $u_-$: $u^\eta(x_0)=u_-(x_0)+\eta > u_-(x_0)$.

Clearly, $u^\eta\in C_b(\overline G)$ and $u^\eta\le g$ on $\partial G$. We need to construct a $u^\eta$-suitable control $\alpha$ for a randomized initial condition $(\tau,\xi)$. Put 
$$ U=B_{\varepsilon/2}(x_0)\cap\{x\in\overline G:\varphi^\eta(x)>u(x)\},\ \ \Gamma=\{\xi\in U\}\in\mathscr F_\tau$$
and define $\alpha^1\in\mathcal A$ by
$$ \alpha^1_t=(a I_\Gamma+\alpha^0_t I_{\Gamma^c})I_{\{t\ge\tau\}},$$
where $\alpha^0$ is a $u$-suitable control for $(\tau,\xi)$. Furthermore, put
\begin{align*}
\tau_0 & = \inf\{t\ge\tau:X^{\tau,\xi,\alpha^1}_t\in \partial B_{\varepsilon/2}(x_0)\}\wedge\sigma^{\tau,\xi,\alpha^1},\\ 
\tau_1 &= \tau_0 I_{\{\tau_0<\sigma^{\tau,\xi,\alpha^1}\}}+(+\infty) I_{\{\tau_0=\sigma^{\tau,\xi,\alpha^1}\}}.  
\end{align*}
On the event $\{\tau_0<\sigma^{\tau,\xi,\alpha^1}\}=\{\tau_1<\infty\}$ trajectories of $X^{\tau,\xi,\alpha^1}$ hit the boundary of $B_{\varepsilon/2}(x_0)$ before they exit the set $\widehat G$ and $\tau_0=\tau_1$. On the event $\{\tau_0=\sigma^{\tau,\xi,\alpha^1}\}$ the trajectories of $X^{\tau,\xi,\alpha}$ exit $\widehat G$ on the stochastic interval $\Lbrack\tau,\tau_0\Rbrack$ and $\tau_1=\infty$.

Define $\alpha\in\mathcal A$ by
$$ \alpha_t=\alpha^1_t I_{\{t\le\tau_0\}} + \alpha^2_t I_{\{t>\tau_1\}},$$
where $\alpha^2$ is a $u$-suitable control for $(\tau_1,\xi_1)$, $\xi_1=X^{\tau,\xi,\alpha^1}_{\tau_1}$. The process $X^{\tau_1,\xi_1,\alpha^2}$ is well-defined since $\xi_1 I_{\{\tau_1<\infty\}}\in \partial B_{\varepsilon/2}(x_0)\cap\overline G$ is bounded. We are going to show that $\alpha$ is a $u^\eta$-suitable control for $(\tau,\xi)$.

From the pathwise uniqueness property we get $X^{\tau,\xi,\alpha}=X^{\tau,\xi,\alpha^1}$ on $\Lbrack\tau,\tau_0\Rbrack$. Moreover, $\sigma^{\tau,\xi,\alpha}=\sigma^{\tau_1,\xi_1,\alpha^2}$ on $\{\tau_1<\infty\}$
and $X^{\tau,\xi,\alpha}=X^{\tau^1,\xi^1,\alpha^2}$ on $\Lbrack\tau_1,\sigma^{\tau_1,\xi_1,\alpha^2}\Rbrack$.  Take a stopping time $\rho\in[\tau,\sigma^{\tau,\xi,\alpha}]$. We have
\begin{align} \label{eq:3.6}
Z^{\tau,\xi,\alpha}_\rho(u^\eta) I_{\{\rho>\tau_1\}}&=I_{\{\rho>\tau_1\}}\int_\tau^{\tau_1} e^{-\beta s} f(X^{\tau,\xi,\alpha^1},\alpha^1_s)\,ds\nonumber\\
&+ I_{\{\rho>\tau_1\}}\left(\int_{\tau_1}^\rho e^{-\beta s} f(X^{\tau_1,\xi_1,\alpha^2},\alpha^2_s)\,ds+ e^{-\beta\rho} u^\eta(X_\rho^{\tau_1,\xi_1,\alpha^2})\right)\nonumber\\
&\ge I_{\{\rho>\tau_1\}}\int_\tau^{\tau_1} e^{-\beta s} f(X^{\tau,\xi,\alpha^1},\alpha^1_s)\,ds+I_{\{\rho>\tau_1\}} Z^{\tau_1,\xi_1,\alpha^2}_\rho(u).
\end{align}
Applying Definition \ref{df:3} to the stopping time
$$ \widehat\rho=\rho I_{\{\rho>\tau_1\}}+\sigma^{\tau_1,\xi_1,\alpha^2}I_{\{\rho\le\tau_1\}}\in [\tau_1,\sigma^{\tau_1,\xi_1,\alpha^2}],$$
we get
\begin{align} \label{eq:3.7}
\mathsf E(Z^{\tau_1,\xi_1,\alpha^2}_\rho(u) I_{\{\rho>\tau_1\}}|\mathcal F_{\tau_1}) & =\mathsf E(Z^{\tau_1,\xi_1,\alpha^2}_{\widehat\rho}(u) I_{\{\rho>\tau_1\}}|\mathcal F_{\tau_1})\ge I_{\{\rho>\tau_1\}} Z^{\tau_1,\xi_1,\alpha^2}_{\tau_1}(u)\nonumber\\
 & = I_{\{\rho>\tau_1\}} e^{-\beta\tau_1}u(\xi_1)=I_{\{\rho>\tau_1\}} e^{-\beta\tau_1}u^\eta(\xi_1).
\end{align}
The last equality follows from the fact that $\xi_1\in\partial  B_{\varepsilon/2}(x_0)\cap\overline G$ on the set $\{\rho>\tau_1\}$ and $u>\varphi^\eta$ on $\partial  B_{\varepsilon/2}(x_0)\cap\overline G\subset S_\varepsilon$. From (\ref{eq:3.6}), (\ref{eq:3.7}) it follows that
\begin{align} \label{eq:3.8}
\mathsf E(Z^{\tau,\xi,\alpha}_\rho(u^\eta) I_{\{\rho>\tau_1\}}|\mathcal F_{\tau_1}) 
&\ge I_{\{\rho>\tau_1\}}\left(\int_\tau^{\tau_1} e^{-\beta s} f(X^{\tau,\xi,\alpha^1},\alpha^1_s)\,ds+e^{-\beta\tau_1}u^\eta(\xi_1)\right)\nonumber\\
&= I_{\{\rho>\tau_1\}} Z^{\tau,\xi,\alpha^1}_{\tau_1}(u^\eta).
\end{align}

If $\{\tau_0=\sigma^{\tau,\xi,\alpha^1}\}$ then $\rho\le\sigma^{\tau,\xi,\alpha^1}$. Hence, $\{\rho>\tau_0\}\subset\{\tau_0<\sigma^{\tau,\xi,\alpha^1}\}$ and 
\begin{equation} \label{eq:3.9}
\{\rho>\tau_0\}=\{\rho>\tau_1\}.
\end{equation}
Moreover, $\tau_0=\tau_1$ on the set (\ref{eq:3.9}). Therefore, by (\ref{eq:3.8}),
\begin{align} \label{eq:3.10}
\mathsf E(Z^{\tau,\xi,\alpha}_\rho(u^\eta)|\mathcal F_{\tau}) &=\mathsf E(I_{\{\rho\le\tau_1\}}Z^{\tau,\xi,\alpha}_\rho(u^\eta)|\mathcal F_{\tau})+\mathsf E(I_{\{\rho>\tau_1\}}\mathsf E(Z^{\tau,\xi,\alpha}_\rho(u^\eta)|\mathcal F_{\tau_1})|\mathcal F_\tau)\nonumber\\
& \ge \mathsf E(I_{\{\rho\le\tau_0\}}Z^{\tau,\xi,\alpha^1}_\rho(u^\eta)|\mathcal F_{\tau})+\mathsf E(I_{\{\rho>\tau_0\}} Z^{\tau,\xi,\alpha^1}_{\tau_0}(u^\eta)|\mathcal F_{\tau})\nonumber\\
& =\mathsf E(Z^{\tau,\xi,\alpha^1}_{\rho\wedge\tau_0}(u^\eta)|\mathcal F_\tau).
\end{align}

On the stochastic interval $\Lbrack\tau_\Gamma,(\rho\wedge\tau_0)_\Gamma\Rbrack$ the trajectories of $X^{\tau,\xi,\alpha^1}$ do not leave the set $B_{\varepsilon/2}(x_0)\cap \overline G$. Indeed, a trajectory $(X_s)_{\tau\le s\le\rho\wedge\tau_0}$, starting from $\Gamma\subset B_{\varepsilon/2}(x_0)$, lies in $\overline G$, and $\tau_0$ occurs before this trajectory exits $B_{\varepsilon/2}(x_0)$. Hence, the estimate $u^\eta(X^{\tau,\xi,\alpha^1}_{\rho\wedge\tau_0})\ge\varphi^\eta(X^{\tau,\xi,\alpha^1}_{\rho\wedge\tau_0})$ holds true on $\Gamma$ and we get the inequality
\begin{equation} \label{eq:3.11}
Z^{\tau,\xi,\alpha^1}_{\rho\wedge\tau_0}(u^\eta)=Z^{\tau,\xi,a}_{\rho\wedge\tau_0}(u^\eta) I_\Gamma+Z^{\tau,\xi,\alpha^0}_{\rho\wedge\tau_0}(u^\eta) I_{\Gamma^c}\ge Z^{\tau,\xi,a}_{\rho\wedge\tau_0}(\varphi^\eta) I_\Gamma+Z^{\tau,\xi,\alpha^0}_{\rho\wedge\tau_0}(u) I_{\Gamma^c}.
\end{equation}

Applying Ito's formula, we obtain
\begin{align*}
Z^{\tau,\xi,a}_{\rho\wedge\tau_0}(\varphi^\eta) &=\int_\tau^{\rho\wedge\tau_0} e^{-\beta s} f(X^{\tau,\xi,a},a)\,ds+e^{-\beta \rho\wedge\tau_0}\varphi^\eta(X^{\tau,\xi,a}_{\rho\wedge\tau_0})\\
&=e^{-\beta\tau}\varphi^\eta(\xi)+\int_\tau^{\rho\wedge\tau_0} e^{-\beta s}\left[f(X_s^{\tau,\xi,a},a)+(\mathcal L^a\varphi^\eta-\beta\varphi^\eta)(X^{\tau,\xi,a}_s)\right]\,ds\\
&+\int_\tau^{\rho\wedge\tau_0}e^{-\beta s}\varphi^\eta_x(X_s^{\tau,\xi,a})\cdot\sigma(X_s^{\tau,\xi,a},a)\,dW_s.
\end{align*}
In view of (\ref{eq:3.5}), taking the conditional expectation we get
\begin{equation} \label{eq:3.12}
\mathsf E(Z^{\tau,\xi,a}_{\rho\wedge\tau_0}(\varphi^\eta) I_\Gamma|\mathscr F_\tau)\ge e^{-\beta\tau}\varphi^\eta(\xi) I_\Gamma=e^{-\beta\tau} u^\eta(\xi) I_\Gamma=Z^{\tau,\xi,\alpha}_\tau(u^\eta) I_\Gamma.\end{equation}
Furthermore,
\begin{equation} \label{eq:3.13}
\mathsf E(Z^{\tau,\xi,\alpha^0}_{\rho\wedge\tau_0}(u) |\mathscr F_\tau)I_{\Gamma^c} \ge Z_\tau^{\tau,\xi,\alpha^0}(u) I_{\Gamma^c}=Z_\tau^{\tau,\xi,\alpha}(u^\eta) I_{\Gamma^c}
\end{equation}
by the definition of $\alpha^0$. By combining (\ref{eq:3.12}), (\ref{eq:3.13}) with (\ref{eq:3.11}) and (\ref{eq:3.10}) we get the desired inequality
$$ \mathsf E(Z^{\tau,\xi,\alpha}_\rho(u^\eta) |\mathscr F_\tau)\ge Z_\tau^{\tau,\xi,\alpha}(u^\eta),$$
which finishes the proof. 

For an interior point $x_0\in G$ the argumentation is almost the same, but slightly simpler and we omit it.  
\end{proof}

\section{Stochastic supersolutions}
\label{sec:4}
\setcounter{equation}{0}
Stochastic supersolutions are defined in the same manner as stochastic subsolutions. However, the notion of a suitable control is not needed here.
\begin{definition} \label{df:4}
With the notation of Section 3 we call $w\in C_b(\overline G)$ a \emph{stochastic subsolution} if $w\ge g$ on $\partial G$ and 
\begin{equation} \label{eq:4.1}
\mathsf E(Z^{\tau,\xi,\alpha}_\rho(w)|\mathscr F_\tau)\le Z^{\tau,\xi,\alpha}_\tau(w)
\end{equation}
for any randomized initial condition $(\tau,\xi)$, control process $\alpha\in\mathcal A$ and stopping time $\rho\in [\tau,\sigma^{\tau,\xi,\alpha}]$.
\end{definition}

The set $\mathcal V^+$ of stochastic supersolutions is non-empty: any sufficiently large constant 
$$ c\ge\max\{\overline f/\beta,\overline g\},\ \ \overline f=\sup_{(x,a)\in\overline G\times A} f(x,a),\ \ \overline g=\sup_{x\in\partial G} g(x)$$
belongs to $\mathcal V^+$ since
$$ \mathsf E(Z^{\tau,\xi,\alpha}_\rho(c)|\mathscr F_\tau)\le\mathsf E\left[(c-\overline f/\beta)(e^{-\beta\rho}-e^{-\beta\tau})+c e^{-\beta\tau}|\mathcal F_\tau\right]\le c e^{-\beta\tau}=Z^{\tau,\xi,\alpha}_\tau(c).$$

Any stochastic supersolution $w$ is an upper bound for $v$. Indeed, for $\tau=0$, $\xi=x$ we have
$$w(x)=Z_0^{x,\alpha}\ge \mathsf E(Z_{\sigma^{x,\alpha}}^{x,\alpha})\ge J(x,\alpha),$$
where the upper index $\tau=0$ is omitted by the previous conventions.

\begin{lemma} \label{lem:3}
Let $w_1$, $w_2$ be stochastic supersolutions. Then $w_1\wedge w_2$ is a stochastic supersolution.
\end{lemma}

The proof follows from the inequality
\begin{align*}
\mathsf E(Z^{\tau,\xi,\alpha}_\rho(w_1\wedge w_2)|\mathcal F_\tau) &\le \min_{i=1,2}\mathsf E(Z^{\tau,\xi,\alpha}_\rho(w_i)|\mathcal F_\tau)\le\min_{i=1,2}Z^{\tau,\xi,\alpha}_\tau(w_i)\\
&= e^{-\beta\tau}(w_1\wedge w_2)(\xi)=Z^{\tau,\xi,\alpha}_\tau(w_1\wedge w_2).
\end{align*}
\begin{lemma} \label{lem:4}
There exists a sequence $w_n\in\mathcal V^+$, $w_n(x)\ge w_{n+1}(x)$, $x\in\overline G$ such that
$$\lim_{n\to\infty} w_n(x)=w_+(x):=\inf\limits_{w\in\mathcal V^+} w(x).$$
\end{lemma}

The proof of Lemma \ref{lem:4} is an evident modification of that of Lemma \ref{lem:2}. The proof of the next theorem is similar to that of Theorem \ref{th:2} and contains some inevitable repetition.

\begin{theorem} \label{th:3}
The function
$$ w_+(x)=\inf\limits_{w\in\mathcal V^+} w(x)$$
is a viscosity subsolution of (\ref{eq:2.3}), (\ref{eq:2.4}).
\end{theorem}
\begin{proof}
Let $x_0\in\partial G$. If $w_+$ is not a viscosity subsolution then there exist $\varphi\in C^2(\overline G)$ and $\varepsilon>0$ such that $w_+(x_0)=\varphi(x_0)$, $w_+<\varphi$ on $(\overline{B_\varepsilon(x_0)}\backslash\{0\})\cap\overline G$ and 
$$ \min\{\beta \varphi(x_0) - H(x_0,\varphi_x(x_0),\varphi_{xx}(x_0)),\varphi(x_0)-g(x_0)\} > 0. $$
By the continuity of $H$, $g$, we can assume that
\begin{equation} \label{eq:4.2}
\varphi(x)>g(x),\ \ x\in B_\varepsilon(x_0)\cap \partial G,
\end{equation}
\begin{equation} \label{eq:4.3}
\beta \varphi(x) - H(x,\varphi_x(x),\varphi_{xx}(x))>0,\ \ x\in B_\varepsilon(x_0)\cap \overline G.
\end{equation}
Furthermore, by the upper-semicontinuity of $w_+$ we have
$$ w_+(x)\le\varphi-\delta,\ \ x\in S_\varepsilon:=\left(\overline B_\varepsilon(x_0)\backslash B_{\varepsilon/2}(x_0)\right)\cap\overline G$$
for some $\delta>0$. In the same way as in the proof of Theorem \ref{th:2}, using Lemma \ref{lem:4}, one can show that there exist $w\in\mathcal V^+$ and $\delta'\in (0,\delta)$ such that $w\le\varphi-\delta'$ on $S_\varepsilon$.

Now take an $\eta\in (0,\delta')$ such that (\ref{eq:4.2}), (\ref{eq:4.3}) hold true for $\varphi^\eta=\varphi-\eta$ instead of $\varphi$. We have $w-\varphi^\eta\le -\delta'+\eta<0$ on $S_\varepsilon$.

To get a contradiction it is enough to prove that the function
$$ w^\eta=\left\{\begin{array}{l l}
    \varphi^\eta\wedge w  & \text{on}\quad B_\varepsilon(x_0)\cap \overline{G},\\
    w & \ \text{otherwise} 
  \end{array} \right. $$
is a stochastic supersolution, since $w^\eta(x_0)=\varphi^\eta(x_0)=w_+(x_0)-\eta < w_+(x_0)$, contrary to the definition of $w_+$.

Since $w\in C_b(\overline G)$ and $w\ge g$ on $\partial G$, we only should to prove that (\ref{eq:4.1}) holds true for any randomized initial condition $(\tau,\xi)$, control process $\alpha\in\mathcal A$ and stopping time $\rho\in [\tau,\sigma^{\tau,\xi,\alpha}]$. Put 
\begin{align*}
\tau_0 &= \inf\{t\ge\tau:X^{\tau,\xi,\alpha}_t\in\partial B_{\varepsilon/2}(x_0)\}\wedge\sigma^{\tau,\xi,\alpha},\\
\tau_1 &= \tau_0 I_{\{\tau^0<\sigma^{\tau,\xi,\alpha}\}} + (+\infty) I_{\{\tau^0=\sigma^{\tau,\xi,\alpha}\}},\ \ \xi_1=X^{\tau,\xi,\alpha}_{\tau_1}.
\end{align*}
The process $X^{\tau_1,\xi_1,\alpha}$ is well-defined since $\xi_1 I_{\{\tau_1<\infty\}}\in\partial B_{\varepsilon/2}(x_0)\cap\overline G$ is bounded.

Similar to (\ref{eq:3.6}) we have
\begin{align} \label{eq:4.4}
Z^{\tau,\xi,\alpha}_\rho(w^\eta) I_{\{\rho>\tau_1\}} \le I_{\{\rho>\tau_1\}}\int_\tau^{\tau_1} e^{-\beta s} f(X^{\tau,\xi,\alpha},\alpha_s)\,ds+I_{\{\rho>\tau_1\}} Z^{\tau_1,\xi_1,\alpha}_\rho(w).
\end{align}
Applying Definition \ref{df:4} to the stopping time
$$ \widehat\rho=\rho I_{\{\rho>\tau_1\}}+\sigma^{\tau_1,\xi_1,\alpha} I_{\{\rho\le \tau_1\}}\in [\tau_1,\sigma^{\tau_1,\xi_1,\alpha}],$$
we obtain
\begin{align} \label{eq:4.5}
\mathsf E(Z^{\tau_1,\xi_1,\alpha}_\rho(w) I_{\{\rho>\tau_1\}}|\mathcal F_{\tau_1}) & =\mathsf E(Z^{\tau_1,\xi_1,\alpha}_{\widehat\rho}(w) I_{\{\rho>\tau_1\}}|\mathcal F_{\tau_1})\le I_{\{\rho>\tau_1\}} Z^{\tau_1,\xi_1,\alpha}_{\tau_1}(w)\nonumber\\
 & = I_{\{\rho>\tau_1\}} e^{-\beta\tau_1}w(\xi_1)=I_{\{\rho>\tau_1\}} e^{-\beta\tau_1}w^\eta(\xi_1).
\end{align}
The last equality follows from the fact that $w<\varphi^\eta$ on $\partial B_{\varepsilon/2}(x_0)\cap \overline G\subset S_\varepsilon$, while $\xi_1\in \partial B_{\varepsilon/2}(x_0)\cap \overline G$ on the set 
\begin{equation} \label{eq:4.6}
\{\rho>\tau_1\}=\{\rho>\tau_0\}.
\end{equation}
From (\ref{eq:4.4}), (\ref{eq:4.5}) it follows that
\begin{align} \label{eq:4.7}
\mathsf E(Z^{\tau,\xi,\alpha}_\rho(w^\eta) I_{\{\rho>\tau_1\}}|\mathcal F_{\tau_1}) 
&\le I_{\{\rho>\tau_1\}}\left(\int_\tau^{\tau_1} e^{-\beta s} f(X^{\tau,\xi,\alpha},\alpha_s)\,ds+e^{-\beta\tau_1}w^\eta(\xi_1)\right)\nonumber\\
&= I_{\{\rho>\tau_1\}} Z^{\tau,\xi,\alpha}_{\tau_1}(w^\eta).
\end{align}

Since, moreover, $\tau_0=\tau_1$ on the set (\ref{eq:4.6}), by (\ref{eq:4.7}) we have
\begin{align} \label{eq:4.8}
\mathsf E(Z^{\tau,\xi,\alpha}_\rho(w^\eta)|\mathcal F_{\tau}) &=\mathsf E(I_{\{\rho\le\tau_1\}}Z^{\tau,\xi,\alpha}_\rho(w^\eta)|\mathcal F_{\tau})+\mathsf E(I_{\{\rho>\tau_1\}}\mathsf E(Z^{\tau,\xi,\alpha}_\rho(w^\eta)|\mathcal F_{\tau_1})|\mathcal F_\tau)\nonumber\\
& \le \mathsf E(I_{\{\rho\le\tau_0\}}Z^{\tau,\xi,\alpha}_\rho(w^\eta)|\mathcal F_{\tau})+\mathsf E(I_{\{\rho>\tau_0\}} Z^{\tau,\xi,\alpha}_{\tau_0}(w^\eta)|\mathcal F_{\tau})\nonumber\\
& =\mathsf E(Z^{\tau,\xi,\alpha}_{\rho\wedge\tau_0}(w^\eta)|\mathcal F_\tau).
\end{align}

Put 
$$ U=B_{\varepsilon/2}(x_0)\cap\{x\in\overline G:\varphi^\eta(x)<w(x)\},\ \ \Gamma=\{\xi\in U\}\in\mathscr F_\tau.$$ 
On the stochastic interval $\Lbrack\tau_\Gamma,(\rho\wedge\tau_0)_\Gamma\Rbrack$ the trajectories of $X^{\tau,\xi,\alpha}$ do not leave the set $B_{\varepsilon/2}(x_0)\cap \overline G$. Hence, we have $w^\eta(X^{\tau,\xi,\alpha}_{\rho\wedge\tau_0})I_\Gamma\le\varphi^\eta(X^{\tau,\xi,\alpha}_{\rho\wedge\tau_0})I_\Gamma$ and
\begin{equation} \label{eq:4.9}
Z^{\tau,\xi,\alpha}_{\rho\wedge\tau_0}(w^\eta)\le Z^{\tau,\xi,\alpha}_{\rho\wedge\tau_0}(\varphi^\eta) I_\Gamma+Z^{\tau,\xi,\alpha}_{\rho\wedge\tau_0}(w) I_{\Gamma^c}.
\end{equation}

Furthermore, by Ito's formula,
\begin{align*}
Z^{\tau,\xi,\alpha}_{\rho\wedge\tau_0}(\varphi^\eta) &=\int_\tau^{\rho\wedge\tau_0} e^{-\beta s} f(X^{\tau,\xi,\alpha},\alpha_s)\,ds+e^{-\beta \rho\wedge\tau_0}\varphi^\eta(X^{\tau,\xi,\alpha}_{\rho\wedge\tau_0})\\
&=e^{-\beta\tau}\varphi^\eta(\xi)+\int_\tau^{\rho\wedge\tau_0} e^{-\beta s}\left[f(X_s^{\tau,\xi,\alpha},\alpha_s)+(\mathcal L^\alpha\varphi^\eta-\beta\varphi^\eta)(X^{\tau,\xi,\alpha}_s)\right]\,ds\\
&+\int_\tau^{\rho\wedge\tau_0}e^{-\beta s}\varphi^\eta_x(X_s^{\tau,\xi,\alpha})\cdot\sigma(X_s^{\tau,\xi,\alpha},\alpha_s)\,dW_s,
\end{align*}
where the infinitesimal generator $\mathcal L^a$ is defined by (\ref{eq:3.3}). 
Taking the conditional expectation, and using (\ref{eq:4.3}), we get
\begin{equation} \label{eq:4.10}
 \mathsf E(Z^{\tau,\xi,\alpha}_{\rho\wedge\tau_0}(\varphi^\eta) I_\Gamma|\mathscr F_\tau)\le e^{-\beta\tau}\varphi^\eta(\xi) I_\Gamma=e^{-\beta\tau} w^\eta(\xi) I_\Gamma=Z^{\tau,\xi,\alpha}_\tau(w^\eta) I_\Gamma.
\end{equation}
Moreover,
\begin{equation} \label{eq:4.11}
 \mathsf E(Z^{\tau,\xi,\alpha}_{\rho\wedge\tau_0}(w) |\mathscr F_\tau)I_{\Gamma^c} \le Z_\tau^{\tau,\xi,\alpha}(w) I_{\Gamma^c}=Z_\tau^{\tau,\xi,\alpha}(w^\eta) I_{\Gamma^c}.
\end{equation}
The desired inequality 
$$ \mathsf E(Z^{\tau,\xi,\alpha}_\rho(w^\eta) |\mathscr F_\tau)\le Z_\tau^{\tau,\xi,\alpha}(w^\eta).$$
follows from (\ref{eq:4.10}), (\ref{eq:4.11}), combined with (\ref{eq:4.9}), (\ref{eq:4.8}).

The case of interior point $x_0\in G$ is simpler and we omit details.  
\end{proof}

\section{Proof of Theorem \ref{th:1} and final remarks}
\label{sec:5}
\begin{proof}[Proof of Theorem \ref{th:1}]
From Theorems \ref{th:2}, \ref{th:3} and inequalities $u\le v\le w$, $u\in\mathcal V^-$, $w\in\mathcal V^+$ it follows that
$$ u_-(x)\le v(x)\le w_+(x),\ \ x\in\overline G,$$
where $u_-$ is a viscosity supersolution and $w_+$ is a viscosity subsolution of (\ref{eq:2.3}), (\ref{eq:2.4}). By the SRC we have $w_+\le u_-$ on $G$. It follows that $u_-(x)=v(x)=w_+(x)$, $x\in G$ and these functions are continuous on $G$. Moreover, by the SRC, $w_+\le \tilde v_*$, $\tilde v^*\le u_-$ on $G$ for any viscosity solution $\tilde v$ of (\ref{eq:2.3}), (\ref{eq:2.4}). Hence, $v=\tilde v$ on $G$, as stated by Theorem 1. 
\end{proof}

Note that $v$ is a viscosity solutions of (\ref{eq:2.3}), but the previous argumentation, in general, does not imply that $v$ satisfies (\ref{eq:2.4}) in the viscosity sense (see Definition \ref{df:1}). 
However, if we assume the \emph{strong uniqueness} property (see \cite{BarSou91}): $\widetilde u(x)\le\widetilde w(x)$, $x\in\overline G$ for any viscosity subsolution $\widetilde u$ and supersolution  $\widetilde w$ of (\ref{eq:2.3}), (\ref{eq:2.4}), then $u_-=v=w_+$ on $\overline G$ and $v$ is a continuous viscosity solution of (\ref{eq:2.3}), (\ref{eq:2.4}). In this case the Dirichlet boundary condition is attained in the classical sense: $v=g$ on $\partial G$. The strong uniqueness property is satisfied if the diffusion matrix does not degenerate along the normal direction to the boundary (see e.g. \cite{BarRou98}, \cite{Cha04}).

One way to prove that $v$ is a viscosity solution of (\ref{eq:2.3}), (\ref{eq:2.4}) in general case is to apply the dynamic programming principle. Although we do not pursue this route here, we present a short proof of the validity of the DPP for $\widehat G=G$. 

Take a stochastic subsolution $u$ and an $u$-suitable control process $\widehat\alpha\in\mathcal A$ for the initial condition $(0,x)$, $x\in\overline G$. By Definition \ref{df:3} we have:
$$ u(x)\le\mathsf E\left[\int_0^\rho e^{-\beta s} f(X_s^{x,\widehat\alpha},\widehat\alpha_s)\,ds+e^{-\beta\rho} u(X_\rho^{x,\widehat\alpha})\right],\ \ 0\le\rho\le\sigma^{x,\widehat\alpha}.$$
In view of the inequality $u\le v$ we get 
$$ u_-(x)=\sup_{u\in\mathcal V^-} u(x)\le\sup_{\alpha\in\mathcal A} \inf_{0\le\rho\le\sigma^{x,\alpha}} \mathsf E\left[\int_0^\rho e^{-\beta s} f(X_s^{x,\alpha},\alpha_s)\,ds+e^{-\beta\rho} v(X_\rho^{x,\alpha})\right]$$
for $x\in\overline G$. On the other hand, for a stochastic supersolution $w$ by Definition \ref{df:4} we have
$$ w(x)\ge \mathsf E\left[\int_0^\rho e^{-\beta s} f(X_s^{x,\alpha},\alpha_s)\,ds+e^{-\beta\rho} w(X_\rho^{x,\alpha})\right],\ \ 0\le\rho\le\sigma^{x,\alpha}, \ \ \alpha\in\mathcal A.$$
With the use of the inequality $v\le w$ we get
$$ w_+(x)=\inf_{w\in\mathcal V^+} w(x)\ge\sup_{\alpha\in\mathcal A} \sup_{0\le\rho\le\sigma^{x,\alpha}} \mathsf E\left[\int_0^\rho e^{-\beta s} f(X_s^{x,\alpha},\alpha_s)\,ds+e^{-\beta\rho} v(X_\rho^{x,\alpha})\right]$$
for $x\in\overline G$.

The equality $u_-(x)=v(x)=w_+(x)$, $x\in G$ now implies that $v$ satisfies the dynamic programming principle:
\begin{align*}
v(x) &=\sup_{\alpha\in\mathcal A} \sup_{0\le\rho\le\sigma^{x,\alpha}} \mathsf E\left[\int_0^\rho e^{-\beta s} f(X_s^{x,\alpha},\alpha_s)\,ds+e^{-\beta\rho} v(X_\rho^{x,\alpha})\right]\\
&=\sup_{\alpha\in\mathcal A} \inf_{0\le\rho\le\sigma^{x,\alpha}}\mathsf E\left[\int_0^\rho e^{-\beta s} f(X_s^{x,\alpha},\alpha_s)\,ds+e^{-\beta\rho} v(X_\rho^{x,\alpha})\right]
\end{align*}
for $x\in G$. But for $x\in\partial G$ these relations are trivially satisfied since $\sigma^{x,\alpha}=0$ in the case $\widehat G=G$.


\begin{thebibliography}{00}
\bibitem{BarBur95}
Barles G., Burdeau J. The Dirichlet problem for semilinear second-order degenerate elliptic equations and applications to stochastic exit time control problems. Comm. Partial Differential Equations. 20 (1995), 129-178.
\bibitem{BarRou98}
Barles G., Rouy E. A strong comparison result for the Bellman equation arising in stochastic exit time control problems and its applications. Comm. Partial Differential Equations. 22 (1998), 1995-2033.
\bibitem{BarSou91}
Barles G., Souganidis P.E. Convergence of approximation schemes for fully nonlinear second order equations. Asymptot. Anal. 4 (1991), 271-283.
\bibitem{BaySir12a}
Bayraktar E., S\^irbu M. Stochastic Perron's method and verification without smoothness using viscosity comparison: the linear case. Proc. Amer. Math. Soc. 140 (2012), 3645-3654.
\bibitem{BaySir12b}
Bayraktar E., S\^irbu M. Stochastic Perron's method for Hamilton-Jacobi-Bellman equations. Preprint arXiv:1212.3170 [math.PR], 21 pages.
\bibitem{BaySir12c}
Bayraktar E., S\^irbu M. Stochastic Perron's method and verification without smoothness using viscosity comparison: obstacle problems and Dynkin games. Preprint arXiv:1112.4904 [math.OC], 13 pages.
\bibitem{Bee93}
Beer G. Topologies on closed and closed convex sets. Kluwer, Dordrecht, 1993.
\bibitem{Bor89}
Borkar V.S. Optimal control of diffusion processes. Pitman Research Notes in Math. 203, Longman, Harlow, UK, 1989. 
\bibitem{Cha04}
Chaumont S. Uniqueness to elliptic and parabolic Hamilton-Jacobi-Bellman equations with non-smooth boundary. C.R. Math. Acad. Sci. Paris. 339 (2004), 555-560.
\bibitem{CraIshLio92}
Crandall M., Ishii H., Lions P.-L. User's guide to viscosity solutions of second-order
partial differential equations. Bull. Amer. Math. Soc. 27 (1992), 1-67.
\bibitem{DonKry07}
Dong H., Krylov N.V. On time-inhomogeneous controlled diffusion processes in domains. Ann. Probab. 35 (2007), 206-227.
\bibitem{FleSon06}
Fleming W.H., Soner H.M. Controlled Markov processes and viscosity solutions. Stochastic Modelling and Applied Probability, 25, Springer, New York, 2nd ed. 2006.
\bibitem{Ish87}
Ishii H. Perron's method for Hamilton-Jacobi equations. Duke Math. J. 55 (1987), 369-384.
\bibitem{Jak10}
Jakobsen E.R. Monotone schemes.  In Encyclopedia of Quantitative Finance, 1253-1263, Wiley, Chichester, 2010. 
\bibitem{Kry77}
Krylov N.V. Controlled diffusion processes. Nauka, Moscow, 1977 (in Russian) [English translation: Springer, New York, 1980].
\bibitem{Lio83a}
Lions P.L. Optimal control of diffusion processes and Hamilton-Jacobi-Bellman equations.
Part 1: the dynamic programming principle and applications. Comm. Partial Differential Equations. 8 (1983), 1101-1174.
\bibitem{Lio83b}
Lions P.L. Optimal control of diffusion processes and Hamilton-Jacobi-Bellman equations.
Part 2: viscosity solutions and uniqueness. Comm. Partial Differential Equations. 8 (1983), 1229-1276.
\bibitem{Sir13}
S\^irbu M. Stochastic Perron's method and elementary strategies for zero-sum differential games. Preprint arXiv:1305.5083 [math.OC], 17 pages.
\bibitem{StrVar72}
Stroock D.W., Varadhan S.R.S. On degenerate elliptic-parabolic operators of second
order and their associated diffusions. Comm. Pure Appl. Math. 25 (1972), 651-713.
\bibitem{Tou13}
Touzi N. Optimal stochastic control, stochastic target problems, and backward SDE. Fields Institute Monographs, 29, Springer, New York, 2013.
\end{thebibliography}
\end{document}